\documentclass[12pt,a4paper]{article}
\usepackage{latexsym,amssymb,amsfonts,amsmath,amsthm,nccmath,float,enumitem,setspace,bm}
\usepackage[usenames,dvipsnames]{xcolor}
\usepackage[hidelinks]{hyperref}
\usepackage[margin=2cm]{geometry}

\hypersetup{
    colorlinks,
    linkcolor={red!80!black},
    citecolor={blue!80!black},
    urlcolor={blue!80!black}
}

\newtheorem{theorem}{Theorem}

\newtheorem{lemma}[theorem]{Lemma}
\newtheorem{conjecture}[theorem]{Conjecture}

\theoremstyle{definition}

\def \deg {{\rm deg}}

\def \av {{\rm av}}
\def \leq {\leqslant}
\def \geq {\geqslant}
\def \le {\leqslant}
\def \ge {\geqslant}

\def \R {\mathbb{R}}

\def \d {\delta}

\def \mod#1{{\:({\rm mod}\ #1)}}

\let\oldproofname=\proofname
\renewcommand{\proofname}{\rm\bf{\oldproofname}}

\title{On the minimum degree required for a triangle decomposition}

\author{Peter J. Dukes\thanks{Department of Mathematics and Statistics, University of Victoria, Canada ({\tt dukes@uvic.ca}); research supported by NSERC grant 312595--2017} \and Daniel Horsley\thanks{School of Mathematics, Monash University, Australia ({\tt danhorsley@gmail.com}); research supported by ARC grants DP150100506 and FT160100048.}}

\date{}

\begin{document}
\setstretch{1.1}
\maketitle

\begin{abstract}
We prove that, for sufficiently large $n$, every graph of order $n$ with minimum degree at least $0.852n$ has a fractional edge-decomposition into triangles. We do this by refining a method used by Dross to establish a bound of $0.9n$. By a result  of Barber, K\"{u}hn, Lo and Osthus, our result implies that, for each $\epsilon >0$, every graph of sufficiently large order $n$ with minimum degree at least $(0.852+\epsilon)n$ has a triangle decomposition if and only if it has all even degrees and number of edges a multiple of three.
\end{abstract}

\section{Introduction}

A \emph{$K_3$-decomposition} of a graph $G$ is a set of triangles in $G$ whose edge sets partition $E(G)$. A \emph{fractional $K_3$-decomposition} of a graph $G$ is an assignment of nonnegative weights to the triangles of $G$ so that, for each edge of $G$, the sum of the weights of all the triangles containing that edge is 1. A $K_3$-decomposition can be viewed as a fractional $K_3$-decomposition in which each assigned weight is 0 or 1.

Obviously for a graph $G$ to have a $K_3$-decomposition, all its degrees must be even and its number of edges must be divisible by 3. We call such graphs \emph{$K_3$-divisible}. Kirkman \cite{Kir1847} showed that every complete graph $K_n$ which is $K_3$-divisible has a $K_3$-decomposition. Such a decomposition is equivalent to a Steiner triple system of order $n$; here, $K_3$-divisibility reduces to the familiar congruence condition $n \equiv 1$ or $3 \pmod{6}$.  Nash-Williams \cite{Nas1970} conjectured that a $K_3$-decomposition exists for every $K_3$-divisible graph with sufficiently high minimum degree. Although he equivocated somewhat on the degree threshold, his conjecture is usually stated as follows.

\begin{conjecture}[\cite{Nas1970}]\label{C:nashWilliams}
Every $K_3$-divisible graph of order $n$ with minimum degree at least $\frac{3}{4}n$ has a $K_3$-decomposition.
\end{conjecture}

For any positive $h \equiv 3 \mod{6}$, the graph $C_4 \cdot K_h$, in which each vertex of a $4$-cycle is blown up into a complete graph of order $h$,
is $(3h-1)$-regular and $K_3$-divisible but can be shown  not to have a $K_3$-decomposition nor even a fractional $K_3$-decomposition. This construction appeared first in Ron Graham's addendum to \cite{Nas1970}, and shows that the value of $\frac{3}{4}$ in Conjecture~\ref{C:nashWilliams} cannot be lowered, even if we weaken the conjecture to demand only fractional $K_3$-decompositions. Here we establish the following.

\begin{theorem}\label{T:fracMain}
There is an integer $N$ such that every graph of order $n > N$ and minimum degree at least $0.852n$ has a fractional $K_3$-decomposition.
\end{theorem}

This result is an improvement on a similar theorem of Dross \cite{Dro}, in which the minimum degree threshold is $0.9n$.   Our proof follows the same general method of pushing triangle weights along $4$-cliques; indeed, our work in essence explores the limits of this approach. This is the latest in a sequence of minimum degree bounds of the form $(1-\delta)n$ sufficient for $K_3$-decompositions, starting with Gustavsson who showed \cite{Gus1991} that one can take $\delta=10^{-24}$, and followed by better values of $\delta$, \cite{Yus2005} then \cite{Gar2014} and finally \cite{Dro}, for the fractional relaxation.
Shortly after this paper appeared as a preprint, Delcourt and Postle \cite{DP} posted a preprint proving that one can take $\delta=0.1727$.

Together with \cite[Theorem 1.3]{BaKuLoOs2016}, Theorem~\ref{T:fracMain} immediately implies the following.

\begin{theorem}\label{T:intMain}
For each real number $\epsilon>0$, there is an integer $N'$ such that every $K_3$-divisible graph of order $n > N'$ and minimum degree at least $(0.852+\epsilon)n$ has a $K_3$-decomposition.
\end{theorem}

In fact, $\epsilon$ can be taken to equal $0$, as we discuss following the proof of Theorem~\ref{T:fracMain}.

The outline of the paper is as follows.  In Section~\ref{S:Dross}, we review the method of \cite{Dro} in detail and set up some notation to be used later.
In particular, we recall a sufficient condition for fractional $K_3$-decomposition of $G$ in terms of the number of $4$-cliques across a partition $(A,B)$ of $E(G)$. We also provide examples that demonstrate that the approach of \cite{Dro} (and likewise ours) cannot by itself solve the problem for $\delta > \frac{1}{6}$. We apply somewhat different strategies depending on the value of two key parameters: the size $|A|$ of one side of our partition and the average over $e \in A$ of the number of triangles containing $e$.  Section~\ref{S:bounds} establishes some basic bounds on these and other parameters.  In Section~\ref{S:outer-alpha}, estimates on crossing $4$-cliques are obtained by convexity arguments inspired by those in \cite{Dro}.  These are generally less effective when our cut $(A,B)$ is close to balanced.  Section~\ref{S:middle-alpha} finishes these remaining cases by classifying vertices according to the number of edges in $A$ induced by their neighbourhoods.

\section{The approach of Dross and a barrier to it}
\label{S:Dross}

For a graph $G$, let $\mathcal{T}(G)$ be the set of all triangles in $G$. For any assignment $\omega$ of weights to the triangles of a graph $G$ and any edge $xy \in E(G)$, we denote by $\omega(xy)$ the sum of $\omega(X)$ over all triangles $X$ in $G$ that contain the edge $xy$. We refer to $\omega(xy)$ as the \emph{weight on the edge $xy$}. For a set $V$, we denote the complete graph on vertex set $V$ by $K_V$. For a graph $G$ and a subset $U$ of $V(G)$, let $G[U]$ denote the subgraph of $G$ induced by $U$.

Let $n \geq 7$ be an integer and $\d$ be a real number such that $0 < \d < 1$. We say that a graph $G$ is an \emph{$(n,\delta)$-reduced graph} if $G$ has order $n$ and minimum degree at least $(1-\d) n$, and each triangle in $G$ has at least one vertex of degree at most $\lceil(1-\d)n+1\rceil$. To prove that every graph of order $n$ and minimum degree at least $(1-\d)n$ has a fractional $K_3$-decomposition, it suffices to show that every $(n,\delta)$-reduced graph has a fractional $K_3$-decomposition. To see this, note that if $G$ and $G'$ are graphs such that $G'$ is obtained from $G$ by deleting the edges of a triangle $X$ in $G$, then a fractional $K_3$-decomposition of $G'$ can be extended to a fractional $K_3$-decomposition of $G$ by simply assigning weight 1 to $X$ and weight 0 to each other triangle in $\mathcal{T}(G) \setminus \mathcal{T}(G')$.

To find a fractional $K_3$-decomposition of an $(n,\delta)$-reduced graph $G$ with $m$ edges, Dross begins by assigning weight $\frac{m}{3|\mathcal{T}(G)|}$ to each triangle in $G$. This means that the sum of the weights on the edges of $G$ is $m$, because each triangle in $G$ contributes its weight to three edges. He then repeatedly uses an elegant switch, which we encapsulate in Lemma~\ref{L:drossSwitch}, to modify this initial assignment of weights until a fractional $K_3$-decomposition of $G$ is obtained. We call a pair of non-adjacent edges $\{ab,cd\}$ in a graph $G$ a \emph{rooted pair} if $G[\{a,b,c,d\}]$ is a copy of $K_4$.

\begin{lemma}[\cite{Dro}]\label{L:drossSwitch}
Let $G$ be a graph, and let $\omega:\mathcal{T}(G) \rightarrow \R$ be an assignment of weights to the triangles of $G$. Let $\epsilon$ be a positive real number, let $ab$ and $cd$ be a rooted pair of edges in $G$, and take a new assignment of weights $\omega':\mathcal{T}(G) \rightarrow \R$ defined by
\[\omega'(X)=
\left\{
  \begin{array}{ll}
    \omega(X)-\frac{\epsilon}{2}, & \hbox{if $X \in \{(a,b,c),(a,b,d)\}$;} \\
    \omega(X)+\frac{\epsilon}{2}, & \hbox{if $X \in \{(a,c,d),(b,c,d)\}$;} \\
    \omega(X), & \hbox{otherwise.}
  \end{array}
\right.\]
Then $\omega'(ab)=\omega(ab)-\epsilon$, $\omega'(cd)=\omega(cd)+\epsilon$, and $\omega'(xy)=\omega(xy)$ for each $xy \in E(G) \setminus \{ab,cd\}$.
\end{lemma}

We will refer to applying Lemma~\ref{L:drossSwitch} as \emph{sending weight $\epsilon$ from $ab$ to $cd$}. Given an initial assignment of weight $\frac{m}{3|\mathcal{T}(G)|}$ to each triangle of a sufficiently dense graph $G$, we can repeatedly apply Lemma~\ref{L:drossSwitch} to adjust the weighting to one in which each edge has weight 1. This assignment will only be a fractional decomposition, however, if we can ensure that the final weight of each triangle is nonnegative.

We introduce some notation that we will employ frequently throughout the remainder of the paper. All of this notation is implicitly dependent on a fixed $(n,\delta)$-reduced graph $G$ that will always be clear from context. We define $m=|E(G)|$. For each edge $e=uv$ of $G$, we let $T_{e}$ be the set of vertices adjacent in $G$ to both $u$ and $v$, and let $t_{e}=|T_{e}|$. For $S \subseteq E(G)$, let $t_S=\frac{1}{|S|}\sum_{e \in S}t_e$ and let $t_\av=t_{E(G)}=\frac{3|\mathcal{T}(G)|}{m}$. Note that the initial weight Dross assigns to each triangle is equal to $\frac{1}{t_\av}$. Let $A$ be a subset of $E(G)$. We abbreviate $\frac{|A|}{m}$ to $\alpha$. We say that a rooted pair $\{e_1,e_2\}$ in $G$ is \emph{separated by} $A$ when $|\{e_1,e_2\} \cap A|=1$ and we define $\kappa_A$ to be the number of rooted pairs in $G$ that are separated by $A$. Finally we let
\[\lambda_A = \tfrac{3}{2}|A|\lceil(1-\d)n-1\rceil(t_A-t_\av).\]

In order to determine where to use Lemma~\ref{L:drossSwitch}, Dross employs an auxiliary flow network. We synthesise the argument in Lemma~\ref{L:flowNetwork}.  For the purposes of the next two lemmas, we define
\[c_{\max}=\mfrac{2}{3t_{\av}\lceil(1-\d)n-1\rceil}.\]

\begin{lemma}[\cite{Dro}]\label{L:flowNetwork}
An $(n,\delta)$-reduced graph $G$ has a fractional $K_3$-decomposition if, for each subset $A$ of $E(G)$ with $t_A > t_\av$, we have $\kappa_A \geq \lambda_A$.

\end{lemma}

\begin{proof}
Let $E^+=\{e \in E(G):t_e > t_{\av}\}$ and $E^-=\{e \in E(G):t_e < t_{\av}\}$. Let $w_{\av}=\frac{1}{t_{\av}}=\frac{m}{3|\mathcal{T}(G)|}$ and let $\omega$ be the weighting of the triangles in $G$ that assigns weight $w_{\av}$ to each triangle. Observe that each edge $e \in E(G)$ has $\omega(e)=t_ew_{\av}$ and that this is $t_ew_{\av}-1$ greater than the desired weight of 1 if $e \in E^+$ and $1-t_ew_{\av}$ smaller than $1$ if $e \in E^-$. Let $z=\sum_{e \in E^+}(t_ew_{\av}-1)$ be the sum of the excess weights on the edges of $E^+$ and note that we also have $z=\sum_{e \in E^-}(1-t_ew_{\av})$ because the sum of the weights on all the edges of $E(G)$ is $m$.

We construct a flow network $N$ on vertex set $E(G) \cup \{s,t\}$, where is $s$ a source and $t$ a sink, whose arcs are given as follows.
\begin{itemize}[nosep]
    \item
For each rooted pair of edges $\{e_1,e_2\}$ in $G$, there are arcs of capacity $c_{\max}$ from $e_1$ to $e_2$ and from $e_2$ to $e_1$.
    \item
For each edge $e \in E^+$ there is an arc of capacity $t_ew_{\av}-1$ from $s$ to $e$.
    \item
For each edge $e \in E^-$ we add an arc of capacity $1-t_ew_{\av}$ from $e$ to $t$.
\end{itemize}
We now prove that if $N$ admits a flow of magnitude $z$, then $G$ has a fractional $K_3$-decomposition.

Suppose that $N$ admits a flow of magnitude $z$. Such a flow uses each arc of $N$ adjacent to either the source or the sink at full capacity and thus, for each $e \in E(G)$, $e$ has a net flow out from it of $t_ew_{\av}-1$ if $e \in E^+$ and $e$ has a net flow into it of $1-t_ew_{\av}$ if $e \in E^-$. Furthermore, the flow from $e_1$ to $e_2$ is at most $c_{\max}$ for any rooted pair $\{e_1,e_2\}$. Let $\omega'$ be the weighting of the triangles in $G$ obtained by beginning with $\omega$ and, for each arc $(e_1,e_2)$ of $N$ such that $e_1,e_2 \in E(G)$, using Lemma~\ref{L:drossSwitch} to shift weight $\epsilon$ from $e_1$ to $e_2$ where $\epsilon$ is the flow along the arc $(e_1,e_2)$. Then we have $\omega'(e)=1$ for each $e \in E(G)$ by the properties of the flow. Furthermore, the weight sent through any rooted pair is at most $c_{\max}$, and each triangle is in at most $\lceil(1-\d)n-1\rceil$ copies of $K_4$ (recall $G$ is $(n,\delta)$-reduced) and hence in at most $3\lceil(1-\d)n-1\rceil$ rooted pairs. So, for each $X \in \mathcal{T}(G)$,
\[\omega'(X) \geq w_{\av}-\tfrac{3}{2}\lceil(1-\d)n-1\rceil c_{\max}=0.\]
Thus $\omega'$ is a fractional $K_3$-decomposition of $G$.

So it suffices to show that $N$ admits a flow of magnitude $z$ if the hypothesis of the lemma is satisfied. By the max-flow min-cut theorem, $N$ admits a flow of magnitude $z$ if and only if the capacity of $N$ across each cut is at least $z$. Let $\{A \cup \{s\},B \cup \{t\}\}$ be a cut of $N$, where $(A,B)$ is a bipartition of $E(G)$. The capacity across this cut is
\begin{equation}\label{E:cutCapacity}
\medop\sum_{e \in A \cap E^-}(1-t_ew_{\av})+\medop\sum_{e \in B \cap E^+}(t_ew_{\av}-1)+\kappa_{A} c_{\max}.
\end{equation}
Now $\sum_{e \in B \cap E^+}(t_ew_{\av}-1)=z-\sum_{e \in A \cap E^+}(t_ew_{\av}-1)$ and hence \eqref{E:cutCapacity} is equal to
\begin{equation}\label{E:cutCapacity2}
z+\kappa_{A}c_{\max} - \medop\sum_{e \in A}(t_ew_{\av}-1)=z+\kappa_{A}c_{\max} - |A|(t_{A}w_{\av}-1).
\end{equation}
If $t_A \leq t_\av$, then $t_Aw_\av \le 1$ and \eqref{E:cutCapacity2} is clearly at least $z$. So we may assume that $t_A > t_\av$. Using $w_{\av}=\frac{1}{t_{\av}}$ and the definition of $c_{\max}$, we see that this last expression is at least $z$ exactly when $\kappa_A \geq \lambda_A$. Thus $N$ admits a flow of magnitude $z$ by the hypotheses of the lemma.
\end{proof}

Roughly speaking, when $G$ has high minimum degree, we expect that it will contain many copies of $K_4$ and hence, for any subset $A$ of $E(G)$, many rooted pairs separated by $A$. That is, $\kappa_A$ will be large and we can hope that $\kappa_A \geq \lambda_A$.
Our overall strategy here is the same as that of \cite{Dro} in that we ultimately use Lemma~\ref{L:flowNetwork} to prove our result. However, we improve the analysis of \cite{Dro} in several ways. Firstly, we obtain stronger bounds on some key parameters of $G$ (see Lemmas~\ref{L:mBound} and \ref{L:tavBound} to follow). Secondly, we extend arguments from \cite{Dro} to produce a new bound on $\kappa_A$ that is particularly effective when $t_A$ is low (see Lemmas~\ref{L:window} and \ref{L:windowCor}). Finally, to deal with the remaining cases, which occur when $\alpha$ is in a middle range, we introduce a new approach to bound $\kappa_A$ (see Section~\ref{S:middle-alpha}).  This approach considers, for each vertex $v$ of $G$, the set of edges in $A$ whose endpoints are both neighbours of $v$ and investigates how these sets intersect.

We now observe that using Lemma~\ref{L:drossSwitch} cannot, by itself, solve the problem for $\delta>\frac{1}{6}$.

\begin{lemma}
\label{L:ConstructionImmuneDross}
For each real $\epsilon>0$, there is an $(n,\delta)$-reduced graph $G$ with $\d<\frac{1}{6}+\epsilon$ for which a fractional $K_3$-decomposition cannot be obtained by first assigning each triangle weight $\frac{1}{t_{\av}}$ and then applying Lemma~\ref{L:drossSwitch} in such a way that each rooted pair in $G$ has weight at most $c_{\max}$ sent through it.
\end{lemma}

\begin{proof}
Let $h$ be a positive integer sufficiently large that $\frac{h+5}{6h+2}<\frac{1}{6}+\epsilon$. We will construct a $(6h+2,\delta)$-reduced graph $G$ with $\d=\frac{h+5}{6h+2}$. Let $G$ be a graph of order $6h+2$ with vertex set $V_1 \cup \cdots \cup V_6 \cup \{u,v\}$ such that
\begin{itemize}[nosep]
    \item
for each $i \in \{1,\ldots,6\}$, $|V_i|=h$ and $G[V_i]$ is empty;
    \item
for each $x \in \{u,v\}$, $x$ is adjacent to each other vertex in $G$; and
    \item
for all distinct $i,j \in \{1,\ldots,6\}$, $G[V_i \cup V_j]$ is isomorphic to the graph obtained from $K_{h,h}$ by removing the edges of a $1$-factor.
\end{itemize}
Then $\deg_G(x)=6h+1$ for each $x \in \{u,v\}$ and $\deg_G(x)=5h-3$ for each $x \in V(G) \setminus \{u,v\}$, and hence $G$ is indeed $(6h+2,\delta)$-reduced with $\d=\frac{h+5}{6h+2}$.

The edge $uv$ is in $6h$ triangles in $G$ and so initially receives weight $\frac{6h}{t_{\av}}$. Furthermore, the edge $uv$ is in $\binom{6}{2}h(h-1)$ rooted pairs. So if each of these has weight at most $c_{\max}$ sent through it, then the final weight of the edge $uv$ will be at least
\[\mfrac{6h}{t_{\av}}-\mbinom{6}{2}h(h-1)c_{\max} = 1+\mfrac{210h^2-269h+89}{3(5h-4)(10h^2-15h+8)}.\]
Because the right hand expression is strictly greater than 1, this proves the result. The equality can be established by noting that $\d=\frac{h+5}{6h+2}$, that $t_{\av}=\frac{3|\mathcal{T}(G)|}{m}$, that $m=15h^2-3h+1$ by a degree sum argument, and that $|\mathcal{T}(G)|=2h(10h^2-15h+8)$ (because there are $\binom{6}{3}h(h-1)(h-2)$ triangles in $G$ that contain neither $u$ nor $v$,  $2\binom{6}{2}h(h-1)$ that contain exactly one of $u$ or $v$, and $6h$ that contain both $u$ and $v$).
\end{proof}

\section{Bounds on parameters of \texorpdfstring{$\bm{G}$}{G}}
\label{S:bounds}

In this section we prove some bounds on $\lambda_A$, $m$, $t_{\av}$ and $t_A$ that will be useful later.

\begin{lemma}\label{L:reqNumBound}
For any $(n,\delta)$-reduced graph $G$ and subset $A$ of $E(G)$,
\begin{itemize}
    \item[(i)]
$\lambda_A \leq \tfrac{3}{2}\alpha(1-\alpha)m\lceil(1-\d)n-1\rceil(t_A-(1-2\d)n)$; and
    \item[(ii)]
$\lambda_A \leq \tfrac{3}{2}(1-\alpha)m\lceil(1-\d)n-1\rceil(t_\av-(1-2\d)n)$
\end{itemize}
\end{lemma}

\begin{proof}
Let $B=E(G) \setminus A$. Because $G$ has minimum degree at least $(1-\d)n$, we have that $t_e \geq (1-2\d)n$ for each $e \in E(G)$ and hence that $t_B \geq (1-2\d)n$. So, because $t_\av = \alpha t_A + (1-\alpha)t_B$,
\begin{equation}\label{E:tavObvious}
t_\av \geq \alpha t_A + (1-\alpha)(1-2\d)n.
\end{equation}
Now (i) follows by using \eqref{E:tavObvious} in the definition of $\lambda_A$. Rewriting \eqref{E:tavObvious} as $t_A \leq \frac{1}{\alpha}(t_\av-(1-\alpha)(1-2\d)n)$ and using this in the definition of $\lambda_A$ produces (ii).
\end{proof}

For convenience we will often slightly weaken these bounds by replacing $\lceil(1-\d)n-1\rceil$ with $(1-\d)n$.

\begin{lemma}\label{L:mBound}
For any $(n,\delta)$-reduced graph $G$,
\begin{itemize}
    \item[(i)]
$\medop\sum_{v \in U}\deg_G(v) < |U|((1-\d)n+2)+\tfrac{1}{2}(\d n -2)^2$ for any $U \subseteq  V(G)$; and
    \item[(ii)]
$m < \left(\mfrac{2-2\d+\d^2}{4}\right)n^2+(1-\d)n+1.$
\end{itemize}
\end{lemma}

\begin{proof}
Note that (ii) follows from (i) by letting $U=V(G)$.  We now prove (i). Let $U \subseteq V(G)$ and put $U^*= \{v \in U:\deg_G(v) \geq (1-\d)n+2\}$. Let $u$ be a vertex in $U$
of largest degree in $G$. If $U^*=\emptyset$, then clearly (i) holds, so we may assume $u \in U^*$. Let $x=|U^* \setminus N_G(u)|$. Clearly $\deg_G(u) \leq n-x$ and hence $\deg_G(v) \leq n-x$ for each $v \in U$. Let $y=|U^* \cap N_G(u)|$ and note that $G[U^* \cap N_G(u)]$ is empty because $G$ is $(n,\delta)$-reduced and hence $G[U^*]$ is triangle-free. So the $y$ vertices in $U^* \cap N_G(u)$ each have degree at most $n-y$. The $|U|-x-y$ vertices in $U \setminus U^*$ each have degree less than $(1-\d)n+2$ by the definition of $U^*$. Thus,
\[\medop\sum_{v \in U}\deg_G(v) < x(n-x)+y(n-y)+(|U|-x-y)((1-\d)n+2).\]
The latter expression is maximised when $x=y=\frac{1}{2}(\d n -2)$, and (i) follows.
\end{proof}

\begin{lemma}\label{L:tavBound}
For any $(n,\delta)$-reduced graph $G$,
\[t_\av \leq 3(1-\d)n- \mfrac{2n-3\d n+2}{m}\mbinom{n}{2}.\]
\end{lemma}

\begin{proof}
Let $V=V(G)$. Let $G^c$ be the complement of $G$ and let $m^c=\binom{n}{2}-m$. For $i \in \{0,1,2,3\}$, let $t_i$ be the number of triangles in $K_V$ that contain exactly $i$ edges of $G$. For a vertex $v \in V$, the number of triangles in $K_V$ that contain two edges incident with $v$ that are not in $G$ is $\binom{\deg_{G^c}(v)}{2}$ and hence
\begin{equation}\label{E:t1Count}
t_1+3t_0=\sum_{v \in V} \mbinom{\deg_{G^c}(v)}{2}.
\end{equation}
For an edge $uv \in E(G^c)$, the number of triangles in $K_V$ that contain $uv$ is $n-2$, and hence
$3t_0+2t_1+t_2=m^c(n-2)$. Thus, using \eqref{E:t1Count},
\begin{equation}\label{E:triComplementCount}
t_0+t_1+t_2=m^c(n-2)+t_0-\sum_{v \in V} \mbinom{\deg_{G^c}(v)}{2}.
\end{equation}

Because $\sum_{v \in V}\deg_{G^c}(v)=2m^c$ and $0 \leq \deg_{G^c}(v) \leq \d n - 1$ for each $v \in V$, $\sum_{v \in V} \binom{\deg_{G^c}(v)}{2}$ is maximised when $\lfloor \frac{2m^c}{\d n -1} \rfloor$ vertices in $G^c$ have degree $\d n-1$ and all but one (or all) other vertices have degree $0$.
Thus $\sum_{v \in V} \binom{\deg_{G^c}(v)}{2} \leq \frac{2m^c}{\d n -1}\binom{\d n -1}{2}=m^c(\d n -2)$. So it follows from \eqref{E:triComplementCount} and $t_0 \geq 0$, that $t_0+t_1+t_2 \geq (1-\d)m^cn$. Thus, using $m^c=\binom{n}{2}-m$,

\[t_\av=\mfrac{3t_3}{m}=\mfrac{3}{m}\left(\mbinom{n}{3}-(t_0+t_1+t_2)\right) \leq \mfrac{3}{m}\left(\mbinom{n}{3}-(1-\d)n\left(\mbinom{n}{2}-m\right)\right).\]
By simplifying this last expression, we obtain the result.
\end{proof}

\section{Low \texorpdfstring{$\bm{t_A}$}{tA} or low/high \texorpdfstring{$\bm{\alpha}$}{alpha}}
\label{S:outer-alpha}

We first give two results which supply bounds on $\kappa_A$. Lemma~\ref{L:dross} is effectively used in \cite{Dro} and Lemma~\ref{L:window} is our own. We then establish some consequences of these bounds for comparison with $\lambda_A$. Lemmas~\ref{L:windowCor}, \ref{L:drossFowardCor} and \ref{L:drossRevCor} show that $\kappa_{A} \geq \lambda_A$ when $t_A$ is not too large, $\alpha$ is small and $\alpha$ is large, respectively.

\begin{lemma}[\cite{Dro}]\label{L:dross}
For any $(n,\delta)$-reduced graph $G$ and subset $S$ of $E(G)$,
\[\kappa_{S} \geq \tfrac{1}{2}|S|t_S(t_S-\d n)-|S|(|S|-1).\]
\end{lemma}

\begin{proof}
Let $e$ be an edge in $S$. Now, $e$ is in at least $\frac{1}{2}t_e(t_e-\d n)$ copies of $K_4$ because $G[T_e]$ must contain at least this many edges. At most $|S|-1$ edges of $G[T_e]$ can be in $S$, and hence $e$ is an edge of at least $r_e=\frac{1}{2}t_e(t_e-\d n)-|S|+1$ rooted pairs separated by $S$. Taking the sum of $r_e$ over all $e \in S$ and using the convexity of $r_e$ in $t_e$, the result follows.
\end{proof}

Later, we use Lemma~\ref{L:dross} with $S$ taking the role of each of our subsets $A$ and $B$ of $E(G)$.

\begin{lemma}\label{L:window}
For any $(n,\delta)$-reduced graph $G$ and subset $A$ of $E(G)$,
\[\kappa_A \geq \tfrac{1}{2}\alpha(1-\alpha)m(t_A-2\delta n)(t_A-3\delta n).\]
\end{lemma}

\begin{proof}
Let $B=E(G) \setminus A$. First suppose, for all $U \subseteq V(G)$ with $|U|=\lceil(1-2\d)n\rceil$, that $|E(G[U]) \cap A| \geq \frac{1}{2}\alpha(t_A-2\d n)(t_A-3\d n)$. For each edge $e \in B$, it follows from our supposition that $e$ is in at least $\frac{1}{2}\alpha(t_A-2\d n)(t_A-3\d n)$ rooted pairs separated by $A$ because $t_{e} \geq \lceil(1-2\d)n\rceil$. So the result follows because $|B| = (1-\alpha)m$.

Now suppose that there is a set $U \subseteq V(G)$ with $|U|=\lceil(1-2\d)n\rceil$ such that $|E(G[U]) \cap A| < \frac{1}{2}\alpha(t_A-2\d n)(t_A-3\d n)$. Let $uv$ be an edge in $A$. Then $|T_{e} \cap U| \geq |T_{e}|-2\d n$ and there must be at least $\frac{1}{2}(|T_{e}|-2\d n)(|T_{e}|-3\d n)$ edges in $G[T_{e} \cap U]$. By our definition of $U$, at most $\frac{1}{2}\alpha(t_A-2\d n)(t_A-3\d n)$ of those edges are in $A$ and hence $uv$ is in at least
\[\tfrac{1}{2}(|T_{e}|-2\d n)(|T_{e}|-3\d n)-\tfrac{1}{2}\alpha(t_A-2\d n)(t_A-3\d n)\]
rooted pairs separated by $A$. Noting that the above expression is convex in $|T_{e}|$ we see that, by taking a sum over the edges in $A$, we are guaranteed that $\kappa_A \geq |A| \cdot \tfrac{1}{2}(1-\alpha)(t_A-2\d n)(t_A-3\d n)$, as required.
\end{proof}

\begin{lemma}\label{L:windowCor}
For an $(n,\delta)$-reduced graph $G$ and subset $A$ of $E(G)$, we have $\kappa_A \geq \lambda_A$ if $\d=0.148$ and $t_A \leq 0.7619n$.
\end{lemma}

\begin{proof}
Take $\d=0.148$. Applying Lemma~\ref{L:window} together with Lemma~\ref{L:reqNumBound}(i), we have $\kappa_A \geq \lambda_A$  whenever
\begin{equation}\label{E:windowCorEq}
(t_A-2\d n)(t_A-3\d n)-3n(1-\d)(t_A-(1-2\d)n)
\end{equation}
is nonnegative. It can be seen that \eqref{E:windowCorEq} is decreasing in $t_A$ for $t_A \leq n$ and so we obtain a lower bound on it by letting $t_A$ be as large as possible.  Then, for $t_A \le 0.7619n$, the resulting quadratic in $n$ is nonnegative for all positive integers $n$.
\end{proof}

\begin{lemma}\label{L:drossFowardCor}
For any $(n,\delta)$-reduced graph $G$ and subset $A$ of $E(G)$, we have $\kappa_A \geq \lambda_A$ if $\d = 0.148$, $\alpha \leq 0.446$ and $n$ is large.
\end{lemma}

\begin{proof}
Take $\d=0.148$. By Lemma~\ref{L:dross} with $S=A$ and Lemma~\ref{L:reqNumBound}(i), we have that $\kappa_{A} \geq \lambda_A$ whenever
\begin{equation}\label{E:drossFowardCorEq}
t_A(t_A-\d n)-2(\alpha m-1)-3n(1-\alpha)(1-\d)(t_A-(1-2\d)n)
\end{equation}
is nonnegative.

As a quadratic in $t_A$, \eqref{E:drossFowardCorEq} is minimised for $t_A=  \frac{1}{2}n(3 - 3 \alpha(1-\d) - 2 \delta)$.  With $\d=0.148$ and $\alpha \le 0.275$, the critical point occurs when $t_A > n$ and it is easily checked that \eqref{E:drossFowardCorEq} is nonnegative using the boundary condition $t_A \le n$ and the bound on $m$ from Lemma~\ref{L:mBound}(ii).  For the rest of the range of $\alpha$, we use the critical point for $t_A$ and the bound on $m$ from Lemma~\ref{L:mBound}(ii).  With this, \eqref{E:drossFowardCorEq} becomes
a quadratic in $n$ with leading coefficient
$$-0.02848 + 0.793336 \alpha - 1.633284 \alpha^2.$$
This is positive for $0.275 \le \alpha \le 0.446$, and so  $\kappa_{A} \geq \lambda_A$ holds for large $n$.
\end{proof}

\begin{lemma}\label{L:drossRevCor}
For any $(n,\delta)$-reduced graph $G$ and subset $A$ of $E(G)$, we have $\kappa_A \geq \lambda_A$ if $\d = 0.148$, $\alpha \geq 0.692$ and $n$ is large.
\end{lemma}

\begin{proof}
Let $B=E(G) \setminus A$ and note that $\kappa_A=\kappa_B$. Applying Lemma~\ref{L:dross} with $S=B$ and noting that $t_B \geq (1-2\d)n$ shows that
$\kappa_{A} \geq \frac{1}{2}|B|(1-2\d)(1-3\d)n^2-|B|(|B|-1)$. Using this together with Lemma~\ref{L:reqNumBound}(ii) and $|B|=(1-\alpha)m$, we have $\kappa_A \geq \lambda_A$ whenever
\begin{equation}\label{E:drossBackCorEq}
(1-2\d)(1-3\d)n^2-2((1-\alpha)m-1) - 3n(1-\d)(t_\av-(1-2\d)n)
\end{equation}
is nonnegative.

Take $\d=0.148$. By calculating the sign of the appropriate partial derivative at each stage, noting that $n$ is large, we can obtain a sequence of lower bounds for \eqref{E:drossBackCorEq} by successively substituting: the bound of Lemma~\ref{L:tavBound} for $t_\av$, then $0.692$ for $\alpha$, and finally the bound of Lemma~\ref{L:mBound}(ii) for $m$. The leading term of the resulting quadratic in $n$ can be seen to be positive.
\end{proof}

\section{High \texorpdfstring{$\bm{t_A}$}{tA} and middle \texorpdfstring{$\bm{\alpha}$}{alpha}}
\label{S:middle-alpha}

In this section we deal with the cases not covered by the previous section, that is, cases where $0.446 < \alpha < 0.692$ and $t_A > 0.7619n$. Throughout this section we assume that $n$ is large and all asymptotic notation is with respect to this. For a graph $G$ and subset $A$ of $E(G)$, let $E_A$ and $e_A$ be the functions with domain $V(G)$ such that $E_A(u)$ is the set of all edges of $G[N_G(u)]$ that are in $A$, and $e_A(u)=|E_A(u)|$.

\begin{lemma}\label{L:edgeNeighProperties}
Let $G$ be an $(n,\d)$-reduced graph and $A$ be a subset of $E(G)$.
\begin{itemize}
    \item[(i)]
For each $u \in V(G)$, $\frac{1}{2}(1-\d)(1-2\d)n^2-(1-\alpha)m \leq e_A(u) \leq \alpha m$.
    \item[(ii)]
$\sum_{u \in V(G)}e_A(u)=\alpha m t_A$.
    \item[(iii)]
For each $U \subseteq V(G)$,
\[\sum_{u \in U}e_A(u) \leq \alpha m |U|-\tfrac{1}{2}\d n\sqrt{2\alpha m}\left(|U|-(1+\tfrac{1}{2}\d) n+\sqrt{2\alpha m}\right)+q(\alpha,\delta,m,n),\]
where $q$ is $O(n^2)$.
\end{itemize}
\end{lemma}

\begin{proof} We prove (i), (ii) and (iii) separately.

\noindent{\bf (i).}
Let $u \in V(G)$. Clearly $e_A(u) \leq \alpha m$ because $E_A(u) \subseteq A$. There are at least $\frac{1}{2}(1-\d)(1-2\d)n^2$ edges in $G[N_G(u)]$ and at most $|B|=(1-\alpha)m$ of these are in $B$. The remainder must be in $E_A(u)$ and hence $e_A(u) \geq \frac{1}{2}(1-\d)(1-2\d)n^2-(1-\alpha)m$.

\noindent{\bf (ii).}
Each edge $e$ in $A$ is counted $t_{e}$ times in the sum $\sum_{u \in V(G)}e_A(u)$. So, by the definition of $t_A$, we have $\sum_{u \in V(G)}e_A(u)=\alpha m t_A$.

\noindent{\bf (iii).}
Let $G^c$ be the complement of $G$ and let $V(G)=\{v_1,\ldots,v_n\}$ where $|N_{G^c}(v_1) \cap U| \leq \cdots \leq |N_{G^c}(v_n) \cap U|$. Note that $\sum_{u \in U}e_A(u) \leq \alpha m |U| - z$, where
\begin{align}
  z &= |\{(u,v_i,v_j): \mbox{$u \in U$, $1 \leq j < i \leq n$, $uv_i \in E(G^c)$, $v_iv_j \in A$}\}| \nonumber\\
  &= \medop\sum_{i=1}^n a_i|N_{G^c}(v_i) \cap U|, \label{E:phiEq}
\end{align}
where $a_i=|\{j \in \{1,\ldots,i-1\}:v_iv_j \in A\}|$ for each $i \in \{1,\ldots,n\}$. We will prove (iii) by establishing that
\begin{equation}\label{E:desiredPhiBound}
z \geq \tfrac{1}{2}\d n\sqrt{2\alpha m}\left(|U|-(1+\tfrac{1}{2}\d) n+\sqrt{2\alpha m}\right)-q(\alpha,\delta,m,n),
\end{equation}
where $q$ is $O(n^2)$.  To set up the proof of \eqref{E:desiredPhiBound}, let $r$ be the greatest integer such that $\binom{r}{2} \leq \alpha m$. Because of our indexing $\{v_1,\ldots,v_n\}$, and subject to $\sum_{i=1}^n a_i = \alpha m$, \eqref{E:phiEq} is minimised when $a_i=i-1$ for each $i \in \{1,\ldots,r\}$ and $a_{r+1}=\alpha m-\tbinom{r}{2}$. Thus, from \eqref{E:phiEq}, we have
\begin{equation}\label{E:phiBound}
z \geq \medop\sum_{i=1}^r(i-1)|N_{G^c}(v_i) \cap U|.
\end{equation}

Now, $\sum_{i=1}^n|N_{G^c}(v_i) \cap U| = \sum_{u \in U}\deg_{G^c}(u)$ and hence, by Lemma~\ref{L:mBound}(i), $\sum_{i=1}^n|N_{G^c}(v_i) \cap U| \geq \d n(|U|-\tfrac{1}{2} \d n)+O(n)$. Also, $|N_{G^c}(v_i) \cap U| \leq \d n$ for each $i \in \{1,\ldots,n\}$.  Subject to these facts, recalling our indexing of $\{v_1,\ldots,v_n\}$, the bound of \eqref{E:phiBound} is minimised when $|N_{G^c}(v_i) \cap U| = \d n$ for $i \in \{r+1,\ldots,n\}$ and
\[|N_{G^c}(v_i) \cap U| = \mfrac{\d n}{r}\big(|U|-\tfrac{1}{2}\d n-n+r+O(1)\big)\]
for $i \in \{1,\ldots,r\}$.
Given this, using the fact that $r=\sqrt{2 \alpha m}+O(1)$ and $|U| \le n$, it can be seen that \eqref{E:desiredPhiBound} follows from  \eqref{E:phiBound}.\qedhere
\end{proof}

\begin{lemma}\label{L:edgeNeighIntersection}
Let $G$ be an $(n,\delta)$-reduced graph, let $A$ be a subset of $E(G)$, and let $B=E(G) \setminus A$.
\begin{itemize}
    \item[(i)]
For each edge $uv \in B$, then there are at least $e_A(u)+e_A(v)-\alpha m$ rooted pairs separated by $A$ that contain $uv$.
    \item[(ii)]
$\displaystyle{\kappa_A \geq \sum_{uv \in B}\big(e_A(u)+e_A(v)\big)-\alpha(1-\alpha)m^2.}$
\end{itemize}
\end{lemma}

\begin{proof}
We first prove (i). Let $uv \in B$. The set of edges of $G[T_{uv}]$ in $A$ is $E_A(u) \cap E_A(v)$, and so $|E_A(u) \cap E_A(v)|$ is the number of rooted pairs separated by $A$ that contain $uv$. By inclusion-exclusion $|E_A(u) \cap E_A(v)| \geq e_A(u)+e_A(v)-\alpha m$ because $|A|=\alpha m$. So (i) holds. By applying (i) to each edge in $B$, and recalling that $|B|=(1-\alpha)m$, we obtain (ii).
\end{proof}

\begin{lemma}\label{L:worstCaseEValues}
Let $G$ be an $(n,\delta)$-reduced graph on vertex set $V$, let $A$ be a subset of $E(G)$ and let $B=E(G) \setminus A$. Then, for $0 < \d \leq \frac{1}{4}$, $2\d+\frac{1}{2}\d^2 < \alpha <1$, and $n$ sufficiently large,
\[\kappa_A \geq \sum_{uv \in B}\big(f^\dag(u)+f^\dag(v)\big)-\alpha(1-\alpha)m^2\]
where $f^\dag: V \rightarrow \R$ is a function such that $|\{v \in V:f^\dag(v)=e_i\}|=n_i+O(1)$ for $i \in \{1,2,3\}$, and
\begin{itemize}[itemsep=0mm]
    \item
$n_2=(1+\tfrac{1}{2}\d) n-\sqrt{2\alpha m}$ and $e_2=\alpha m$;
    \item
$e_1=\alpha m-\tfrac{1}{2}\d n\sqrt{2\alpha m}$;
    \item
$e_0=\frac{1}{2}(1-\d)(1-2\d)n^2-(1-\alpha)m$;
    \item
$n_0=\frac{1}{e_1-e_0}\big((n-n_2)e_1-\alpha m(t_A-n_2)\big)$ and $n_1=n-n_0-n_2$.
\end{itemize}
Moreover, $0 < e_0 < e_1 < e_2$ and $2e_1 > e_0+e_2$.
\end{lemma}

\begin{proof}
We first show that $0 < e_0 < e_1 < e_2$ and $2e_1 > e_0+e_2$. Obviously $e_1 < e_2$ because $\alpha$ and $\delta$ are positive, and $e_0>0$ using the bound of Lemma~\ref{L:mBound}(ii), $\alpha > 2\d+\frac{1}{2}\d^2$ and the fact that $n$ is large. Because $e_1 < e_2$, showing that $2e_1 > e_0+e_2$ will also establish that $e_0 < e_1$. Routine manipulation shows that $2e_1 > e_0+e_2$ provided that
\[m-\d n\sqrt{2\alpha m}- \tfrac{1}{2}(1-\d)(1-2\d)n^2\]
is positive. This can be seen to be the case by considering the expression as a quadratic in $\sqrt{m}$ and noting $0 < \d \leq \frac{1}{4}$.

For each $v \in V(G)$, let $\deg_B(v)$ denote the number of edges in $B$ that are incident with $v$. Let $V=\{v_1,\ldots,v_n\}$ where $\deg_{B}(v_1) \leq \cdots \leq \deg_{B}(v_n)$. Let $\mathcal{E}$ be the set of all functions $f$ from $V$ to $\R$ that obey the following conditions:
\begin{itemize}[itemsep=0mm]
    \item[(i)]
$e_0 \leq f(v_i) \leq e_2$ for each $i \in \{1,\ldots,n\}$;
    \item[(ii)]
$\sum_{i=1}^nf(v_i)=e_2t_A$;
    \item[(iii)]
$\sum_{i=1}^s f(v_i) \leq se_2-\tfrac{1}{2}\d n\sqrt{2e_2}\left(s-(1+\tfrac{1}{2}\d) n+\sqrt{2e_2}\right)+q(\alpha,\delta,m,n)$ for each $s \in \{1,\ldots,n\}$,
where $q$ is the $O(n^2)$ function given in Lemma~\ref{L:edgeNeighProperties}(iii).
\end{itemize}
Note that $e_A \in \mathcal{E}$ by Lemma~\ref{L:edgeNeighProperties}. For any function $f \in \mathcal{E}$, let
\[\sigma(f)=\sum_{u \in V}f(u)\deg_{B}(u)-\alpha(1-\alpha)m^2.\]
Let $\sigma_{\min}$ be the minimum value of $\sigma(f)$ over all functions $f \in \mathcal{E}$. Note that $\kappa_A \geq \sigma(e_A) \geq \sigma_{\min}$ by Lemma~\ref{L:edgeNeighIntersection}(ii) and the definition of $\sigma_{\min}$. Let $f^\dag$ be a function in $\mathcal{E}$ such that
\begin{itemize}[itemsep=0mm]
    \item[(a)]
$\sigma(f^\dag)=\sigma_{\min}$;
    \item[(b)]
of all the functions in $\mathcal{E}$ obeying (a), $f^\dag$ is one for which the tuple $(f^\dag(v_1),\ldots,f^\dag(v_n))$ is lexicographically maximal.
\end{itemize}
Because $\kappa_A \geq \sigma_{\min}=\sigma(f^\dag)$, to prove the lemma it only remains to show that $|(f^\dag)^{-1}(e_i)|=n_i+O(1)$ for $i \in \{1,2,3\}$.

Let $\epsilon$ be an arbitrarily small positive real number and let $k$ be an arbitrary element of $\{1,\ldots,n-1\}$. Let $f^\ddag_{k}$ be the function from $V$ to $\R$ such that $f^\ddag_{k}(v_{k})=f^\dag(v_{k})+\epsilon$, $f^\ddag_{k}(v_{k+1})=f^\dag(v_{k+1})-\epsilon$ and $f^\ddag_{k}(v_i)=f^\dag(v_i)$ for each $i \in \{1,\ldots,n\}\setminus\{k,k+1\}$. Note that $\sigma(f^\ddag_{k})= \sigma(f^\dag)-\epsilon(\deg_{B}(v_{k+1})-\deg_{B}(v_{k}))$ and hence $\sigma(f^\ddag_{k}) \leq \sigma(f^\dag)$ by our indexing of $\{v_1,\ldots,v_n\}$. Thus, $f^\ddag_{k}$ cannot be in $\mathcal{E}$, for otherwise it would violate either (a) or (b) of the definition of $f^\dag$. Since $f^\ddag_{k}$ clearly obeys (ii), $f^\ddag_{k}$ must violate (i) or (iii).

From the previous paragraph we can make the key observation that, for any $i \in \{1,\ldots,n-1\}$, either $f^\dag(v_{i+1})=e_0$ or  $f^\dag(v_i)=h_i$ where $h_i$ is the minimum of $e_2$ and the bound on $f^\dag(v_i)$ implied by (iii) with $s=i$ given the values of $f^\dag(v_1),\ldots,f^\dag(v_{i-1})$. Let $x$ be the unique integer such that the bound of (iii) is at least $se_2$ for $s \leq  x$, but is less than $se_2$ for $s \geq  x+1$. Note that $x=n_2+O(1)$. Then $h_i=e_2$ for $i \in \{1,\ldots,x\}$, $e_1 \leq h_i \leq e_2$ for $i =x+1$, and $h_i=e_1$ for $i \in \{x+2,\ldots,n\}$. Let $y$ be the smallest element of $\{1,\ldots,n\}$ such that $f^\dag(v_{y})=e_0$ (or let $y=n+1$ if no such integer exists). By inductively applying our observation with $i=y,\ldots,n-1$, noting that $e_0 < e_1 \leq h_i$, we can conclude that  $f^\dag(v_{i})=e_0$ for each $i \in \{y,\ldots,n\}$. Next, inductively applying our observation with $i=1,\ldots,y-2$, we can conclude that $f^\dag(v_{i})=h_i$ for each $i \in \{1,\ldots,y-2\}$ and hence, by our comments on $h_i$, that $f^\dag(v_{i})=e_2$ for $i \in \{1,\ldots,x\}$ and $f^\dag(v_{i})=e_1$ for $i \in \{x+2,\ldots,y-2\}$.

So we have established that $f^\dag$ maps $x$ vertices to $e_2$, $n-y+1$ vertices to $e_0$ and all but at most two of the remaining vertices to $e_1$. Thus, given that $x=n_2+O(1)$, it follows from the fact that $f^\dag$ obeys (ii) that
\[(n-y)e_0+(y-n_2)e_1+n_2e_2=e_2t_A+O(n^2).\]
So we can calculate that $n-y=n_0+O(1)$ and $y-x=n_1+O(1)$. This completes the proof.
\end{proof}

We remark that $n_0$ is a rational expression in $n,\delta,t_A$ and $\sqrt{\alpha m}$.  Moreover, since $e_0 < e_1$, the expression is defined and smooth in the parameters.

\begin{lemma}
\label{L:middleAlpha}
For any $(n,\delta)$-reduced graph $G$ and subset $A$ of $E(G)$, we have $\kappa_A \geq \lambda_A$ if $\d=0.148$, $0.446 \le \alpha \le 0.692$, $t_A \ge 0.7619n$, and $n$ is sufficiently large.
\end{lemma}

\begin{proof}
Let $B=E(G) \setminus A$ and, for $i \in \{0,1,2\}$, let $e_i$ and $n_i$ be defined as in Lemma~\ref{L:worstCaseEValues}. Let $f^\dag$ be the function given by Lemma~\ref{L:worstCaseEValues}. For $i \in \{0,1,2\}$, let $V_i=\{v \in V(G):f^\dag(v)=e_i\}$ and $n'_i=|V_i|$, and note that $n'_i=n_i+O(1)$ by Lemma~\ref{L:worstCaseEValues}. The following gives a classification of pairs $\{u,v\}$ of distinct vertices of $G$ according to their values of $f^\dag(u)+f^\dag(v)$.
\begin{itemize}[itemsep=0mm]
    \item
The $\binom{n'_0}{2}$ pairs in $\{\{u,v\}: u,v\in V_0\}$ each have $f^\dag(u)+f^\dag(v)=2e_0$.
    \item
The $n'_0n'_1$ pairs in $\{\{u,v\}: u \in V_0, v \in V_1\}$ each have $f^\dag(u)+f^\dag(v)=e_0+e_1$.
    \item
The $n'_0n'_2$ pairs in $\{\{u,v\}: u \in V_0, v \in V_2\}$ each have $f^\dag(u)+f^\dag(v)=e_0+e_2$.
    \item
All but at most $O(n)$ of the remaining pairs have $f^\dag(u)+f^\dag(v) \geq 2e_1$.
\end{itemize}
Observe that $2e_0<e_0+e_1<e_0+e_2<2e_1$ from Lemma~\ref{L:worstCaseEValues}. Let $b= |B| = (1-\alpha)m$. Then $B$ contains $b$ of the pairs we classified above and from our discussion so far it can be seen that
\begin{equation}\label{E:gFuncDefn}
\sum_{uv \in B}\big(f^\dag(u)+f^\dag(v)\big) \geq g(\alpha,t_A,m)+O(n^3),
\end{equation}
where
\[g(\alpha,t_A,m)=
\left\{
  \begin{array}{ll}
    g_1(\alpha,t_A,m) & \hbox{if $b \leq \frac{1}{2}n_0^2;$} \\
    g_2(\alpha,t_A,m) & \hbox{if $\frac{1}{2}n_0^2<b \leq n_0(\frac{1}{2}n_0+n_1);$} \\
    g_3(\alpha,t_A,m) & \hbox{if $n_0(\frac{1}{2}n_0+n_1)<b \leq n_0(\frac{1}{2}n_0+n_1+n_2);$} \\
    g_4(\alpha,t_A,m) & \hbox{if $b>n_0(\frac{1}{2}n_0+n_1+n_2),$}
  \end{array}
\right.\]
and $g_1,g_2,g_3,g_4$ are the functions defined by
\begin{align*}
  g_1(\alpha,t_A,m) &= 2be_0 \\
  g_2(\alpha,t_A,m) &= n_0^2e_0+(b-\tfrac{1}{2}n_0^2)(e_0+e_1) \\
  g_3(\alpha,t_A,m) &= n_0^2e_0+n_0n_1(e_0+e_1)+\big(b-n_0(\tfrac{1}{2}n_0+n_1)\big)(e_0+e_2) \\
  g_4(\alpha,t_A,m) &= n_0^2e_0+n_0n_1(e_0+e_1)+n_0n_2(e_0+e_2)+2 \big(b-n_0(\tfrac{1}{2}n_0+n_1+n_2)\big)e_1.
\end{align*}
Note that the $O(n^3)$ term in \eqref{E:gFuncDefn} allows us to neglect the $O(1)$ differences between $n_i$ and $n'_i$ for $i \in \{0,1,2\}$ and also the $O(n)$ edges not covered by our classification above.

By \eqref{E:gFuncDefn} and Lemma~\ref{L:worstCaseEValues}, we have $\kappa_A>\lambda_A$ provided that
\begin{equation}
\label{E:limit}
\lim_{n \rightarrow \infty} \frac{g(\alpha,t_A,m) - k(\alpha,t_A,m)}{n^4}>0,
\end{equation}
where
\[k(\alpha,t_A,m)=\alpha(1-\alpha)m \big( m+\tfrac{3}{2}n(1-\d)(t_A-(1-2\d)n)\big)\]
comes from the $\alpha(1-\alpha)m^2$ term in Lemma~\ref{L:worstCaseEValues} and our upper bound on $\lambda_A$ from Lemma~\ref{L:reqNumBound}(i). It is enough to compare the $O(n^4)$ terms in both $g$ and $k$.
The difference $g-k$ is a piecewise differentiable function in the parameters $\alpha,\mu,\tau$, where
$t_A=\tau n$ and $m=\mu \binom{n}{2}=\frac{1}{2}\mu n^2+O(n)$. From our hypotheses we have $0.446 \le \alpha \le 0.692$ and $\tau \geq 0.761$. From Lemma~\ref{L:mBound}(ii), taking $\d=0.148$, we have $0.852 \leq \mu \leq 0.863$. By \eqref{E:tavObvious}, because $\alpha \geq 0.446$, we have $t_{\av} \geq 0.446t_A+0.554(1-2\d)n$. Furthermore, from Lemma~\ref{L:tavBound}, substituting the bound of Lemma~\ref{L:mBound}(ii), we obtain $t_{\av} \leq \frac{2-6\d+9\d^2-3\d^3}{2-2\d+\d^2} n +O(1)$. Combining these two inequalities and solving for $t_A$, we see that $\tau \leq 0.814$. So our parameters take values in the box
\[\Xi = \{(\alpha,\tau,\mu):
0.446 \le \alpha \le 0.692,~
0.761 \le \tau \le 0.814,~
0.852 \le \mu \le 0.863\}.\]

It is not hard to obtain strong numerical evidence that \eqref{E:limit} holds for all $(\alpha,\tau,\mu) \in \Xi$ and hence for the truth of this lemma. Below we give a rigorous computer-assisted verification that \eqref{E:limit} holds for all $(\alpha,\tau,\mu) \in \Xi$.
For this verification, we invoke the following procedure:
\begin{enumerate}
\item
check that, for some positive constant $\rho$, the stronger estimate $g-k > \rho n^4$ holds at each combination of the parameters on a discrete grid $\Xi_h \subset \Xi$ having sub-interval width $h$;
\item
obtain an upper bound on the gradient norms $||\nabla g_i||$, $i=1,2,3,4$ and $||\nabla k||$ over $\Xi$.
\end{enumerate}
Here, gradients are with respect to $\alpha,\tau,\mu$.
Note in particular that, even though $g$ is piecewise defined on $\Xi$, step 2 above actually gives that each $g_i$ (and of course $h$) is well-behaved on the entire box $\Xi$.

Now, as long as
$\rho n^4/h > \max_i ||\nabla g_i|| + ||\nabla k||$,
the mean value theorem ensures that $g-k >0$ on $\Xi$. We carried out step 1 over $\Xi_h$ with $h=0.00001$ and $\rho = 0.00022$. Hence it suffices to show that $\max_i ||\nabla g_i|| + ||\nabla k|| \leq 22n^4+ o(n^4)$ for all  $(\alpha,\tau,\mu) \in \Xi$.

Using Mathematica to symbolically optimise $||\nabla k||$ we have that $||\nabla k|| \leq 0.187$. For the bounds on the gradients $||\nabla g_i||$, we first compute bounds on the leading terms of the constituent functions and their gradients. We present a summary of results in Table~\ref{T:DetailedBounds}. With the exception of the bounds on $||\nabla n_i||$ for $i \in \{0,1\}$, these were again obtained by using Mathematica to symbolically maximise the norms of the gradients.
\begin{table}[H]
\[
\begin{array}{r|ccc}
i & 0 & 1 & 2  \\
\hline
|e_i| n^{-2} \le 	      &0.169	&0.242 	&0.299\\
||\nabla e_i|| n^{-2} \le &0.513 	&0.501 	&0.554\\
|n_i| n^{-1} \le 	      &0.448	&0.490 	&0.458\\
||\nabla n_i|| n^{-1} \le & 7.589	&9.372 	&0.783\\
\end{array}
\hspace{1cm}
\begin{array}{rc}
|b| n^{-2} \le 	&0.240	\\
||\nabla b|| n^{-2} \le & 0.513 \\
\end{array}\]
\caption{Bounds on the leading terms of the constituent functions and gradients in $\Xi$.}
\label{T:DetailedBounds}
\end{table}
To obtain the bound on $||\nabla n_0||$, we considered our expression for $n_0$ as a quotient with numerator $n_0^{\rm n}=(n-n_2)e_1-\alpha m(t_A-n_2)$ and denominator $n_0^{\rm d}=e_1-e_0$. A calculation shows that $n_0^{\rm n}=\frac{1}{2}\alpha \mu(1-\d-\tau)+\frac{1}{4}\d^2\sqrt{\alpha \mu}+o(n^3)$.
We again used Mathematica to show that $|n_0^{\rm n}|n^{-3} \leq 0.0315$, $|n_0^{\rm d}|n^{-2} \geq 0.0692 $, and  $||\nabla n_0^{\rm d}||n^{-2} \leq 0.477$, and
\[||\nabla n_0^{\rm n}||n^{-3} \leq ||\nabla(\tfrac{1}{2}\alpha \mu(1-\d-\tau))||+ ||\nabla(\tfrac{1}{4}\d^2\sqrt{\alpha \mu})|| \leq 0.308.\] Using the quotient rule then gives
\[
||\nabla n_0|| \leq \mfrac{||\nabla n_0^{\rm n}||}{|n_0^{\rm d}|}+\mfrac{|n_0^{\rm n}|\,||\nabla n_0^{\rm d}||}{|n_0^{\rm d}|^2} \leq 7.589n + o(n).
\]
From this, because $n_1=n-n_0-n_2$, we have
\[
||\nabla n_1|| \leq 1+||\nabla n_0||+||\nabla n_2|| \leq 9.372n + o(n).
\]

So, with the bounds in Table~\ref{T:DetailedBounds} now established, we can now bound $||\nabla g_i||$ for $i \in \{1,2,3,4\}$.  Using the chain rule and triangle inequality,
\begingroup
\allowdisplaybreaks
\begin{align*}
||\nabla g_1|| &\leq 2\bigl(|b|\,||\nabla e_0|| + |e_0|\,||\nabla b||\bigr) \\
&\leq 0.420 n^4 + o(n^4),\\
&\\
||\nabla g_2||  &\le |n_0|\bigl(|n_0|\,||\nabla e_0||+2|e_0|\,||\nabla n_0||\bigr) +\bigl(|b|+\tfrac{1}{2}|n_0|^2\bigr)(||\nabla e_0||+||\nabla e_1||)\\
&\mathrel{\phantom{\leq}}+\bigl(|e_0|+|e_1|\bigr)\bigl(||\nabla b||+|n_0|||\nabla n_0||\bigr)\\
& \leq 3.206n^4+ o(n^4),\phantom{\big)}\\
&\\
||\nabla g_3|| &\leq |n_0|\bigl(|n_0|\,||\nabla e_0||+2|e_0|\,||\nabla n_0||\bigr)+|n_0|\,|n_1|\bigl(||\nabla e_0||+||\nabla e_1||\bigr)\phantom{\Big)}\\
&\mathrel{\phantom{\leq}} + \bigl(|e_0|+|e_1|\bigr)\bigl(|n_0|\,||\nabla n_1||+|n_1|\,||\nabla n_0||\bigr) +\Bigl(|b|+|n_0|\bigl(\tfrac{1}{2}|n_0|+|n_1|\bigr)\Bigr)\bigl(||\nabla e_0||+||\nabla e_2||\bigr)\\
&\mathrel{\phantom{\leq}}+\bigl(|e_0|+|e_2|\bigr)\Bigl(||\nabla b||+|n_0|\bigl(\tfrac{1}{2}||\nabla n_0||+||\nabla n_1||\bigr)+\bigl(\tfrac{1}{2}|n_0| +|n_1|\bigr)||\nabla n_0||\Bigr)\\
& \leq 10.863 n^4+ o(n^4),\phantom{\Big)}\text{and}\\
&\\
||\nabla g_4||  &\leq |n_0|\bigl(|n_0|\,||\nabla e_0||+2|e_0|\,||\nabla n_0||\bigr)+|n_0|\,|n_1|\bigl(||\nabla e_0||+||\nabla e_1||\bigr)\phantom{\Big)}\\
&\mathrel{\phantom{\leq}} + \bigl(|e_0|+|e_1|\bigr)\bigl(|n_0|\,||\nabla n_1||+|n_1|\,||\nabla n_0||\bigr)+|n_0|\,|n_2|\bigl(||\nabla e_0||+||\nabla e_2||)\phantom{\Big)}\\
&\mathrel{\phantom{\leq}} + \bigl(|e_0|+|e_2|\bigr)\bigl(|n_0|\,||\nabla n_2||+|n_2|\,||\nabla n_0||\bigr)+\Bigl(2|b|+|n_0|\bigl(|n_0|+2|n_1|+2|n_2|\bigr)\Bigr)||\nabla e_1||\\
&\mathrel{\phantom{\leq}}+|e_1|\Bigl(2||\nabla b||+|n_0|\bigl(||\nabla n_0||+2||\nabla n_1||+2||\nabla n_2||\bigr)+\bigl(|n_0| +2|n_1| +2|n_2|\bigr)||\nabla n_0||\Bigr)\\
&\leq 15.083n^4+ o(n^4).\phantom{\Big)}
\end{align*}
\endgroup
So $\max_i ||\nabla g_i|| + ||\nabla k|| \leq 22n^4+ o(n^4)$ for all  $(\alpha,\tau,\mu) \in \Xi$ as required and our verification is complete.
\end{proof}

We are now able to complete the proof of our main result.

\begin{proof}[\textbf{\textup{Proof of Theorem~\ref{T:fracMain}.}}]
By Lemma~\ref{L:flowNetwork}, it suffices to show that $\kappa_A \geq \lambda_A$ for each subset $A$ of $E(G)$. When $t_A \leq 0.7619n$, this is established by Lemma~\ref{L:windowCor}. When $t_A > 0.7619n$, this is established by Lemma~\ref{L:drossFowardCor} for $\alpha \leq 0.446$, by Lemma~\ref{L:drossRevCor} for $\alpha \geq 0.692$, and by Lemma~\ref{L:middleAlpha} for $0.446 < \alpha < 0.692$.
\end{proof}

As a concluding remark, it is straightforward to check that Lemmas~\ref{L:mBound}(ii), \ref{L:windowCor}, \ref{L:drossFowardCor} and \ref{L:drossRevCor} lead to continuous bounds on the parameters $\alpha,m,t_A$ in a neighbourhood of $\delta=0.148$.   And \eqref{E:limit} was verified in Lemma~\ref{L:middleAlpha} as a strict inequality.  Since the $g_i(\alpha,t_A,m)$ and $k(\alpha,t_A,m)$ are all continuous in an open set slightly larger than $\Xi$, it follows that the unsightly $\epsilon$ can be eliminated in Theorem~\ref{T:intMain}.

\section*{Acknowledgements}

The authors are grateful for the careful reading of the referees, which helped to clarify and tighten up the presentation in a few places.

\end{document}